\documentclass[a4paper]{amsart}

\usepackage[T1]{fontenc}
\usepackage[utf8]{inputenc}
\usepackage[english]{babel}
\usepackage{lmodern}
\usepackage{exscale}
\usepackage[babel]{microtype}
\usepackage{amsmath, amssymb, mathtools, mathrsfs,mathabx,epsfig}
\usepackage{enumerate}
\usepackage{scalerel,stackengine}
\usepackage{color}

\usepackage{csquotes}
\usepackage[backend=biber,style=alphabetic, sorting=nyt, maxnames=10, maxalphanames=5]{biblatex}
\usepackage{tikz}
\usetikzlibrary{matrix,shapes,arrows,calc,3d,decorations,decorations.pathmorphing,through,cd}
\usepackage[numbered]{bookmark}
\usepackage{hyperref}
\usepackage{amsthm}
\hypersetup{
  colorlinks,
  urlcolor = {blue!80!black},
  linkcolor = {red!70!black},
  citecolor = {green!50!black}
}

\theoremstyle{plain}
  \newtheorem{thm}{Theorem}
  \newtheorem{theorem}[thm]{Theorem}
  \newtheorem{definition}[thm]{Definition}

  \newtheorem{proposition}[thm]{Proposition}

  \newtheorem{lemma}[thm]{Lemma}
\theoremstyle{definition}
  
  \newtheorem{remark}[thm]{Remark}

\tikzset{ext/.style={circle, draw,inner sep=1pt},int/.style={circle,draw,fill,inner sep=1pt},nil/.style={inner sep=1pt}}
\tikzset{exte/.style={circle, draw,inner sep=3pt},inte/.style={circle,draw,fill,inner sep=3pt}}
\tikzset{diagram/.style={matrix of math nodes, row sep=3em, column sep=2.5em, text height=1.5ex, text depth=0.25ex}}
\tikzset{diagram2/.style={matrix of math nodes, row sep=0.5em, column sep=0.5em, text height=1.5ex, text depth=0.25ex}}

\newcommand{\alg}[1]{\mathfrak{{#1}}}
\newcommand{\ad}{{\text{ad}}}

\newcommand{\R}{{\mathbb{R}}}

\newcommand{\Graphs}{{\mathsf{Graphs}}}

\newcommand{\Pois}{{\mathsf{Pois}}}

\newcommand{\op}{\mathcal}

\newcommand{\SO}{\mathrm{SO}}

\newcommand{\bpm}{\begin{pmatrix}}
\newcommand{\epm}{\end{pmatrix}}

\newcommand{\GC}{\mathrm{GC}}

\newcommand{\mU}{\mathcal{U}}

\DeclareMathOperator{\sgn}{sgn}

\DeclareMathOperator{\Hom}{Hom}
\DeclareMathOperator{\Emb}{Emb}

\newcommand{\beq}[1]{\begin{equation}\label{#1} 
}
\newcommand{\eeq}{\end{equation}}

\usepackage{fullpage}

\begin{document}

\title{A(nother) model for the framed little disks operad}

\author[E. Lindell]{Erik~Lindell}
\address{Erik~Lindell: University of Stockholm, Matematik, 106 91 Stockholm, Sweden}
\email{lindell@math.su.se}
\author[T. Willwacher]{Thomas~Willwacher}
\address{Thomas~Willwacher: Department of Mathematics, ETH Zurich, Zurich, Switzerland}
\email{thomas.willwacher@math.ethz.ch}
\thanks{This work was partially supported by the European Research Council Starting Grant GRAPHCPX}
\subjclass[2000]{16E45; 53D55; 53C15; 18G55}
\keywords{}
\maketitle

\begin{abstract}
We describe new graphical models of the framed little disks operads which exhibit large symmetry dg Lie algebras.
\end{abstract}

%\tableofcontents

\newcommand{\GT}{\mathrm{GT}}
\newcommand{\GRT}{\mathrm{GRT}}
\newcommand{\Gal}{\mathrm{Gal}}
\newcommand{\lD}{\mathrm{D}}
\newcommand{\Aut}{\mathrm{Aut}}
\newcommand{\Der}{\mathrm{Der}}
\newcommand{\Chains}{\mathrm{Chains}}
\newcommand{\K}{\mathbb K}
\newcommand{\FreeOp}{\mathrm{FreeOp}}

\newcommand{\rell}{\mathfrak{r}^{ell}}
\newcommand{\frakt}{\mathfrak{t}}
\newcommand{\barfrakt}{\overline{\mathfrak{t}}}
\newcommand{\grtell}{\alg{grt}^{ell}}
\newcommand{\grt}{\alg{grt}}
\newcommand{\conf}{\mathrm{conf}}
\newcommand{\Diff}{\mathrm{Diff}}
\newcommand{\BiDer}{\mathrm{BiDer}}

\newcommand{\icg}{\mathsf{ICG}}
\newcommand{\flD}{\mathsf{D}^{fr}}
\newcommand{\tGC}{\mathsf{tGC}}
\newcommand{\osp}{\mathfrak{osp}}

\newcommand{\mflD}{\mathsf{mD}^{fr}}
\newcommand{\mfM}{\mathsf{mM}^{fr}}
\newcommand{\vspan}{\mathrm{span}}

\newcommand{\fG}{\mathsf{fG}}
\newcommand{\actson}{\circlearrowright}
\newcommand{\LG}{\mathsf{LG}}
\newcommand{\LLG}{\mathsf{LLG}}
\newcommand{\hLG}{\widehat{\mathsf{LG}}}

\newcommand{\LGraphs}{\mathsf{LGraphs}}
\newcommand{\oGC}{\mathsf{oGC}}

\newcommand{\GG}{\mathsf{GG}}
\newcommand{\G}{\mathrm{G}}

\newcommand{\hag}{\hat{\alg g}}

\def\acts{\mathrel{\reflectbox{$\righttoleftarrow$}}}

%%%%additional package for tikz
\usetikzlibrary{decorations.markings}

\tikzset{->-/.style={decoration={
  markings,
  mark=at position .5 with {\arrow{>}}},postaction={decorate}}}

%\usetikzlibrary{decorations.pathreplacing,calc}

\section{Introduction}
\label{sec:introduction}

The framed little $n$-disks operad $\flD_n$ is the operad of embeddings (by rotation and scaling) of $n$-dimensional unit disks in a unit disk.
It plays an important role in algebraic topology. For example it appears prominently in the Goodwillie-Weiss embedding calculus \cite{GW1999b,BoavidaWeiss2013}, and in factorization homology \cite{AF2015}.
The real homotopy type of the framed little disks operad was computed for $n=2$ by Giansiracusa-Salvatore \cite{GiansiracusaSalvatore2010} and \v Severa \cite{Severa2010}, and for $n\geq 3$ by the second author and Khoroshkin \cite{KW}.
Recall that a Hopf cooperad is a cooperad in the category of differential graded commutative algebras.
A real model for the framed little disks operad is then a Hopf cooperad which is quasi-isomorphic to the (homotopy) Hopf cooperad of differential forms $\Omega(\flD_n)$.
For example, for even $n$ a small such a model is provided by the cohomology $H(\flD_n)$, using the formality of the operad.
Alternative combinatorially defined models are provided in \cite{GiansiracusaSalvatore2010,Severa2010,KW}.

The motivation behind the present paper is the realization that the framed little disks operad (over $\R$) has a large group of homotopy automorphisms, which are not visible as automorphisms in the existing (Hopf cooperad) models.\footnote{From \cite{KW} one can also extract homotopy Hopf cooperad models which support all symmetries found here, but these models are homotopy Hopf cooperads, not honest Hopf cooperads, and their strictification is relatively unwieldy. }
More concretely, let $\GC_n$ be the Kontsevich graph complex, which is a dg Lie algebra. It carries a grading by loop order, and hence one can extend it to a Lie algebra $\R L\ltimes \GC_n$ where $L$ denotes the generator of the loop order grading.
Then consider the dg Lie algebra
\[
 \alg a_n' := \Der(H(B\SO(n)) \ltimes \left(H(B\SO(n)) \hat \otimes (\R L\ltimes \GC_n) \right).
\]
It contains a Maurer-Cartan element $m\in H(B\SO(n))\hat \otimes \GC_n$ describing the homotopy type of the action of $\SO(n)$ on the unframed little disks operad, see \cite[Theorems 1.1,1.3]{KW} or section \ref{sec:KWmodel} below. We define the dg Lie algebra 
\[
 \alg a_n := (\alg a_n')^m
\]
as the twist of $\alg a_n'$.

The main purpose of this note is to introduce a dg Hopf cooperad $\LG_n$ built from suitably decorated ``Feynman'' diagrams, which is a model for the framed little disks operad and carries an action of the dg Lie algebra $\alg a_n$, compatible with the Hopf cooperad structure.
It will be shown in the forthcoming work \cite{BW} that $\alg a_n$ in fact may be identified with the homotopy (bi)derivations of the (real model of the) framed little disks operad.
Hence, being slightly imprecise, the main result of the present work is to describe a model for the framed little disks operad in which all of its (infinitesimal) homotopy automorphisms can be realized as (infinitesimal) Hopf cooperad automorphisms.

\subsection*{Acknowledgements}
This paper is based on the MSc thesis of the first author \cite{ErikMSc}, written at ETH Zurich under joint supervision of the second author and Alexander Berglund.
Both authors thank Alexander Berglund for his support.

\section{Preliminaries}
\subsection{Notation and conventions}

In this paper we will consistently be working over the reals, so for example cohomology will always be implicitly taken with $\mathbb{R}$-coefficients, and vector spaces are $\R$-vector spaces etc. 
For a set $S$ we denote the vector space generated by $S$ by $\R S$. As usual we use square square brackets for the graded commutative algebra of polynomials in some specified set of generators, e.g., $\R[x,y,z]$.
%We will also only consider real homotopy, i.e. $\pi^{\mathbb{R}}(G):=\pi(G)\otimes\mathbb{R}$.
We abbreviate the phrase ``differential graded'' by dg as usual. Throughout the paper we shall use cohomological conventions unless otherwise stated, so differentials always have degree $+1$. We will denote the cohomology of a dg vector space or a topological space by $H(\dots)$. The homology of a topological space is denoted by $H_\bullet(\dots)$.
We use the notation $V[-n]$ for the vector space where we raise the degree of every element by $n$. This is explicitly constructed as $V[-n]= \mathbb{R}s\otimes V$, where $s$ is some generator of degree $+n$. If we have some element $x$ in a graded vector space $V$, we will therefore write $sx$ for the corresponding element in $V[-n]$, when it is important to distinguish the two. In the same setting, if $x$ denotes an element of $V[-n]$, we denote the corresponding element of $V$ by $s^{-1}x$. Otherwise, we will often abuse notation a bit and simply write $x$ for both elements.

We denote the universal enveloping algebra of a (dg) Lie algebra $\alg g$ by $\mU \alg g$, and dually, the universal enveloping coalgebra of a dg Lie coalgebra $\alg c$ by $\mU^c \alg c$.
For an augmented algebra $A$ with augmentation $\epsilon$ we shall denote the augmentation ideal (i.e., the kernel of the augmentation) by $\bar A$. We use the same notation $\bar C$ for the cokernel of the coaugmentation of a coaugmented coalgebra $C$.

\subsection{Operads and cooperads}

We freely use the language of operads and cooperads, referring the reader to \cite{Loday} for an introduction and definitions.
In particular, we are interested in operads in the symmetric monoidal category of dg cocommutative coalgebras, which we call Hopf operads.
Dually, we will consider cooperads in dg commutative algebras, which we call Hopf cooperads.

% In particular a symmetric sequence $\mathcal{P}$ in the category $\mathcal{C}$ is a sequence of objects
% $$(\mathcal{P}(0),\mathcal{P}(1),\ldots),$$
% indexed by the natural numbers, together with a right action by the symmetric group $S_r$ on each object $\mathcal{P}(r)$, which we for $\mu\in\mathcal{P}(r)$  and $\sigma\in S_r$ denote by $\mu^\sigma$. 

Let $\op T$ be an operad in the category of smooth manifolds. (One can in fact take any topological operad, but less us assume that the underlying spaces are manifolds for simplicity.)
Then one can consider the collection of dg commutative algebras of de Rham forms $\Omega(\op T)$.
The operadic composition then induces a (homotopy) Hopf cooperad structure on $\Omega(\op T)$, referring to \cite{KW} for the appropriate definition of homotopy Hopf cooperad.
We will say that a Hopf cooperad is a model for $\op T$, if it can be connected to $\Omega(\op T)$ by a zigzag of quasi-isomorphisms.

In this paper we are interested in models for the framed little $n$-disks operad. We will, however, not need to work with the topological framed little disks operad directly, but merely construct Hopf cooperads that can be shown to be quasi-isomorphic to the models for the framed little $n$-disks operad found in \cite{KW}.

Suppose that $\mathcal H$ is a cocommutative Hopf algebra acting on a Hopf operad $\op P$.
Then one defines the semi-direct product $\mathcal{P}\circ\mathcal{H}$ as the Hopf operad assembled from the spaces
$$(\mathcal{P}\circ\mathcal{H})(r)=\mathcal{P}(r)\otimes\mathcal{H}^{\otimes r},$$
with the composition defined by
$$(p;h_1,\ldots,h_r)\circ_j(q;k_1,\ldots,k_s):=(p\circ_j(h_j'\cdot q);h_1,\ldots,h_{j-1},h_j''\cdot k_1,\ldots h_j^{(s+1)}\cdot k_s,h_{j+1},\ldots,h_r).$$
Dually, for a commutative Hopf algebra $\mathcal H$ coacting on a Hopf cooperad $\op C$ the collection $\op{C}\circ\mathcal{H}$ naturally carries the structure of a Hopf cooperad.
More concretely, consider the cocomposition correspoding to the tree
\beq{equ:treeex}
T=
\begin{tikzpicture}[scale=.5,baseline=-1cm]
\draw node {} 
  child { 
    child { node {$1$} }
    child {  node{$\cdots$} }
    child {  
      child { node {$i$} }
      child { node{$\cdots$} }
      child { node {$j$} }
    }
    child {node{$\cdots$}}
    child { node{$r$} }
  };
\end{tikzpicture}
\eeq
If the cocomposition maps some $x\in \op C$ to $\Delta_T x = \sum x' \otimes x''$, and if the coaction sends $x''$ to $\sum h\otimes x'''\in H\otimes \op C$, then the cocomposition in the semidirect product is as follows:
\beq{equ:twistcocomp}
\Delta_T (x,h_1,\dots,h_r) \mapsto 
\sum \pm (x',h_1,\dots, h_i'\cdots h_j' h, \dots ,h_k)\otimes (x''',h_i'',\dots, h_j'').
\eeq

\subsection{The original little $n$-disks operads}

The little $n$-disks operad $\lD_n$ was introduced by Boardman, Vogt and May in the study of iterated loop spaces in the early 1970's. Let us recall its definition.
\begin{definition} We assemble $\lD_n$ from the subspaces
$$\lD_n(r)\subset\Emb\left(\underset{r}{\bigsqcup}D^n,D^n\right),$$
of embeddings of $r$ $n$-dimensional disks into the $n$-disk itself, where we assume that the embeddings are \textit{rectilinear}, i.e., obtained by scaling and translating the little disks. The identity is the identity embedding of the disk into itself, while the symmetric action is given by permuting the embeddings. The composition $(f_1,\ldots,f_r)\circ_j(g_1,\ldots,g_s)$ is given by composing with the $j$th embedding $f_j$
$$(f_1,\ldots,f_r)\circ_j(g_1,\ldots,g_s)=(f_1,\ldots,f_{j-1},f_j\circ g_1,\ldots,f_j\circ g_s,f_{j+1},\ldots,f_r).$$
\end{definition}

A typical element of $\mathcal{D}_2(3)$  can thus be illustrated like:
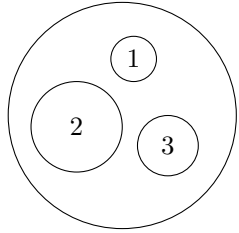
\[\begin{tikzpicture}[scale=1,
vert/.style={draw,outer sep=0,inner sep=0,minimum size=3,shape=circle,fill},
helper/.style={outer sep=0,inner sep=0,minimum size=3,shape=coordinate},
default edge/.style={draw},
 every loop/.style={out=140, in=50, looseness=.8, distance=.8cm }]

\draw (0,0) circle (1.5cm);
\draw (0.15,0.75) circle (0.3cm);
\draw (-0.6,-0.15) circle (0.6cm);
\draw (0.6,-0.4) circle (0.4cm);

\node (l1) at (0.15,0.75){1};
\node (l2) at (-0.6,-0.15){2};
\node (l3) at (0.6,-0.4){3};
\end{tikzpicture}\]

The symmetric action may then be viewed as "permuting the labels" on the disks. The $j$th partial composition can be illustrated by insertion into the $j$th disk. Once again, it is easiest to illustrate this with an example in the 2-dimensional case: 

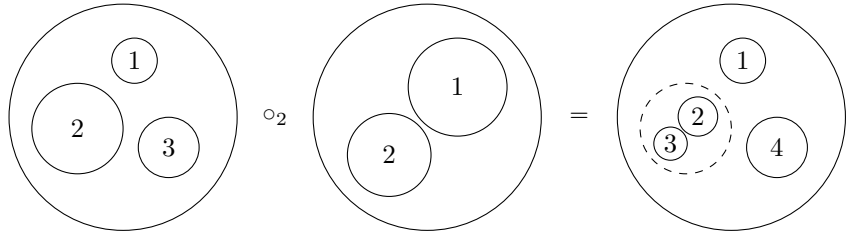
\[\begin{tikzpicture}[scale=1,
vert/.style={draw,outer sep=0,inner sep=0,minimum size=3,shape=circle,fill},
helper/.style={outer sep=0,inner sep=0,minimum size=3,shape=coordinate},
default edge/.style={draw},
 every loop/.style={out=140, in=50, looseness=.8, distance=.8cm }]

\draw (0,0) circle (1.5cm);
\draw (0.15,0.75) circle (0.3cm);
\draw (-0.6,-0.15) circle (0.6cm);
\draw (0.6,-0.4) circle (0.4cm);

\node (l1) at (0.15,0.75){1};
\node (l2) at (-0.6,-0.15){2};
\node (l3) at (0.6,-0.4){3};
\node (l4) at (2,0){$\circ_2$};

\draw (4,0) circle (1.5cm);
\draw (4.4,0.4) circle (0.65cm);
\draw (3.5,-0.5) circle (0.55cm);

\node (l5) at (4.4,0.4){1};
\node (l6) at (3.5,-0.5){2};
\node (l7) at (6,0){$=$};

\draw (8,0) circle (1.5cm);
\draw (8.15,0.75) circle (0.3cm);
\draw (7.4,-0.15)[dashed] circle (0.6cm);
\draw (7.56,0.01) circle (0.26cm);
\draw (7.2,-0.35) circle (0.22cm);
\draw (8.6,-0.4) circle (0.4cm);

\node (l8) at (8.15,0.75){1};
\node (l9) at (7.56,0.01){2};
\node (l10) at (7.2,-0.35){3};
\node (l11) at (8.6,-0.4){4};

\end{tikzpicture}\]

The cohomology of the framed little disks operads has been computed by Arnold \cite{Arnold1969} and Cohen \cite{Cohen1976} (cf. also \cite{Sinha2013}, and shown to agree with the dual of the $n$-Poisson operad $\Pois_n$ for $n\geq 2$. 

\subsection{The framed little $n$-disks operad} If we loosen the criterion that the embeddings in the operad need to be rectilinear, and also permit rotations, we get the \textit{framed} little $n$-disks operad. We can thus represent an element by one from the original operad, together with an element from $SO(n)$ associated to each embedding: $(p;g_1,\ldots,g_r),$ where $p\in\lD_n(r)$ and $g_1,\ldots,g_r\in SO(n)$. For $n=2$ we can illustrate this with the example 
\\
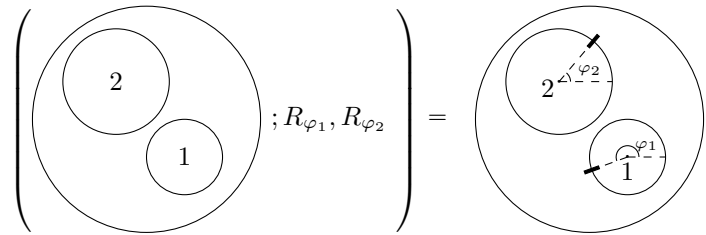
\[\left(\begin{tikzpicture}[scale=1,baseline=-1mm,x=1cm,y=1cm]
\draw (-3,0) circle (1.5cm);
\draw (-2.5,-0.5) circle (0.5cm);
\draw (-3.4,0.5) circle (0.7cm);
\node (l1) at (-2.5,-0.5){1};
\node (l2) at (-3.4,0.5){2};
\node at (-0.6,0){$;R_{\varphi_1},R_{\varphi_2}$};
\end{tikzpicture}\right)
\begin{tikzpicture}[scale=1,baseline=-1mm]
\node at (1,0){$=$};
\draw (3,0) circle (1.5cm);
\draw (3.5,-0.5) circle (0.5cm);
\draw (2.6,0.5) circle (0.7cm);

\node at (3.5,-0.5){$\boldsymbol{\cdot}$};
\node at (3.5,-0.7){1};
\draw[dashed] (3.5,-0.5)--(4,-0.5);
\draw[dashed] (3.5,-0.5)--(3.03,-0.671);
\draw (3.65,-0.5) arc (0:200:0.15cm);
\node at (3.75,-0.35){\tiny{${\varphi_1}$}};

\node at (2.6,0.5){$\boldsymbol{\cdot}$};
\node at (2.45,0.4){2};
\draw[dashed] (2.6,0.5)--(3.3,0.5);
\draw[dashed] (2.6,0.5)--(3.05,1.036);
\draw (2.75,0.5) arc (0:50:0.15cm);
\node at (3,0.65){\tiny{${\varphi_2}$}};

\draw[fill=black,rotate around={200:(3.5,-0.5)}] (3.9,-0.52) rectangle (4.1,-0.48);
\draw[fill=black,rotate around={50:(2.6,0.5)}] (3.2,0.52) rectangle (3.4,0.48);

\end{tikzpicture}\]
\\
where $R_\varphi$ is the rotation matrix associated to the angle $\varphi$. The symmetric action and identity are the same as in $\lD_n$ (with the addition that we permute the associated rotations accordingly with the labels), but when composing disks, we also need to compose rotations. One way of defining this structure is by taking the semi-direct product with relation to the action by $SO(n)$ on $\lD_n$ given by rotating the centers of the embedded disks around the center of the big disk, but keeping the orientations of the little disks themselves fixed. Note that if these are rotated as well, the resulting element is not an element in the operad, since the embeddings are required to be rectilinear. In the case $n=2$, we for example have:
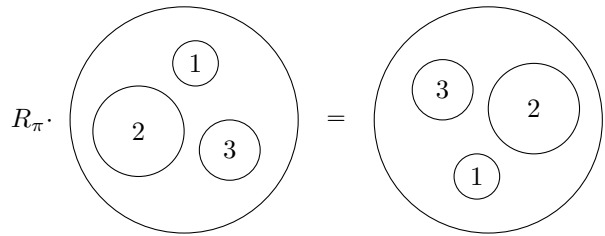
\[\begin{tikzpicture}
 \node at (-2,0){$R_\pi\cdot$} ;
 
\draw (0,0) circle (1.5cm);
\draw (0.15,0.75) circle (0.3cm);
\draw (-0.6,-0.15) circle (0.6cm);
\draw (0.6,-0.4) circle (0.4cm);

\node (l1) at (0.15,0.75){1};
\node (l2) at (-0.6,-0.15){2};
\node (l3) at (0.6,-0.4){3};

\node at (2,0){$=$};

\draw (4,0) circle (1.5cm);
\draw[rotate around={180:(4,0)}](4.15,0.75) circle (0.3cm);
\draw[rotate around={180:(4,0)}](3.4,-0.15) circle (0.6cm); \draw[rotate around={180:(4,0)}](4.6,-0.4) circle (0.4cm);

\node (l1) at (3.85,-0.75){1};
\node (l2) at (4.6,0.15){2};
\node (l3) at (3.4,0.4){3};

\end{tikzpicture}\]

The composition is then given by
$$(p;g_1,\ldots,g_r)\circ_j(q;h_1,\ldots,h_s)=(p\circ_j (g_j\cdot q);g_1,\ldots,g_{j-1}, g_jh_1,\ldots, g_jh_s,g_{j+1},\ldots,g_r),$$
just as desired. The symmetric action and identity are also the same as in $\lD_n$, so we define:

\begin{definition} The framed little $n$-disks operad is the semi-direct product
$$\lD_n^{fr}:=\lD_n\circ \SO(n),$$
with relation to the action by $\SO(n)$ on $\lD_n$ defined above.
\end{definition}

% Since $\lD_n^{fr}=\lD_n\circ \SO(n)$, it follows that the homology of the framed little $n$-disks operad is $\Pois_n\circ H_\bullet(SO(n))$, with relation to the induced action by $H_\bullet(SO(n))$ on $\Pois_n$. In arity $r$ the homology is thus $\Pois_n(r)\otimes H_\bullet(SO(n))^{\otimes r}$. We can view an element of $H_\bullet(\lD_n^{fr})(r)$ as a tree from $\mathsf{Pois}(r)$, with its leaves decorated by elements from $H_\bullet(SO(n))$, so for example
% \begin{align*}
% \begin{tikzpicture}[baseline=-3ex]
% \draw (0,-0.5)--(0,0);
% \node at (0,0.2){$\cdot$};
% \draw (-0.3,0.3)--(-0.8,0.8);
% \draw (0.3,0.3)--(0.8,0.8);
% \node at (0.8,1){$[\ ,\ ]$};
% \draw (1.1,1.1)--(1.4,1.4);
% \draw (0.5,1.1)--(0.2,1.4);
% \node at (-0.8,1){$x_1$};
% \node at (0.2,1.6){$x_2$};
% \node at (1.4,1.6){$x_3$};
% \end{tikzpicture}
% ,
% \end{align*}
% 
% where $x_1,x_2$ and $x_3$ lie in $H_\bullet(SO(n))$. The composition is given by combining grafting with the induced action.

\subsection{The special orthogonal group}\label{sec:son}
We will consider throughout the remainder of the paper the Lie group $G:=\SO(n)$.
Furthermore, we will denote by $\alg g=\pi_{\geq 2}(G)\otimes \R$ the real homotopy groups of $\SO(n)$, considered as an abelian graded Lie algebra.
Note in particular that $\alg g$ is not the Lie algebra associated to $G$ (i.e., $\mathfrak{so}_n$).
The (real) cohomology of $G$ is well known to be
\[
H(G)=H(\SO(n))\cong
\begin{cases}
\R[ p_3,p_7,\ldots,p_{2n-3}] & \text{for $n$ odd,}\\
\R[p_3,p_7,\ldots,p_{2n-5},E_{n-1}] & \text{for $n$ even}
\end{cases},
\]
where the subscripts indicate the degrees of the corresponding generators.
Note that $H(\SO(n))$ is a commutative and cocommutative Hopf algebra, where the stated generators can be chosen to be primitive.
The real homotopy groups of $\SO(n)$ can be identified with the dual of the vector space spanned by those elements, i.e.,
\[
 \alg g = 
\begin{cases}
\R\{ p_3^*,p_7^*,\ldots,p_{2n-3}^*\} & \text{for $n$ odd,}\\
\R\{p_3^*,p_7^*,\ldots,p_{2n-5}^*,E_{n-1}^*\} & \text{for $n$ even}
\end{cases}.
\]
Let us also recall the cohomology of the classifying space $BG$
\[
H(BG)=H(B\SO(n))\cong
\begin{cases}
\R[ p_4,p_8,\ldots,p_{2n-2}] & \text{for $n$ odd,}\\
\R[p_4,p_8,\ldots,p_{2n-4},E_{n}] & \text{for $n$ even}
\end{cases},
\]
where again the subscripts indicate the cohomological degree of the generators.
Note that the algebras $H(BG)$ and $H(G)$ are Koszul dual. We shall use the Koszul complex 
\[
 K = (H(G)\otimes H(BG), d).
\]
Here the Koszul differential $d$ is defined on an element $x\otimes y$ as
\[
 d(x\otimes y) = \sum (-1)^{x'} x' \otimes (sx'')y,
\]
where we use Sweedler notation $x\mapsto \sum x'\otimes x''$ to denote the coproduct of the Hopf algebra $H(G)$.
The Koszul complex is a dgca, and by Koszulness of the algebras involved its cohomology is one-dimensional, spanned by the algebra unit
\[
 H(K)= \R.
\]

Furthermore, note that $H(G)$ may be considered the universal enveloping coalgebra of $\alg g^*$
\[
 H(G) = \mU^c \alg g^*
\]
and similarly $H(BG)$ can be identified with the Chevalley-Eilenberg complex.

We shall also need quasi-free resolutions of $\alg g$. The minimal resolution is the cobar construction
\beq{equ:hagdef}
 \hag = \Omega(H_\bullet(BG)),
\eeq
which agrees with the free Lie algebra generated by the coaugmentation coideal $\bar H_\bullet(BG)[-1]$, with a differential defined using the cocommutative coproduct on that space.
We similarly consider the graded dual $\hag^*$, and the universal enveloping algebras and coalgebras
\begin{align}
\label{equ:hatHdef}
 \hat H_\bullet(G) &:= \mU\hag & \hat H(G) &:= \mU^c\hag^*.
\end{align}
We note in particular that $\hat H(G)$ is the space of words in letters (a basis of) $\bar H(BG)[1]$, with the shuffle product and deconcatenation coproduct.
We can use our resolutions to define a resolution of the Koszul complex $K$ above,
\beq{equ:hatKdef}
 \hat K := (\hat H(G)\otimes H(BG), d).
\eeq
Here the Koszul differential $d$ is defined on a word $w_1\cdots w_r\in \hat H(G)$ (with letters $w_1,\dots,w_r\in \bar H(BG)[1]$) and $y\in H(BG)$ by
\[
 d(w_1\cdots w_r\otimes y) = (-1)^{|w_1|+\cdots+|w_{r-1}|} w_1\cdots w_{r-1}\otimes (sw_r) y.
\]

\section{Graph complexes}\label{sec:graphcpx}

A graph complex is a dg vector space whose elements are linear combinations or series of some sort of graphs.
In this section, we introduce two examples of graph complexes; the complexes (dg Lie algebras) $\GC_n$ and the Hopf cooperads $\GG_n$, which were both defined by Kontsevich. We recall here the construction, following roughly \cite{Willwacher2014,DolgushevWillwacher2015}.

First, let $\mathsf{dgra}_{N,k}$ be the set of directed graphs with $N$ vertices and $k$ directed edges.
More concretely such a graph for us is given by two functions (source and target map) 
\[
 s,t : \{1,\dots, k\}\to \{1,\dots, N\}.
\]
Let $$\mathsf{dgra}_{N,k}^{conn}\subset \mathsf{dgra}_{N,k}$$ be the subset of connected graphs.
Let $S_2^{\times k}$ act on the set of edges by flipping their direction. Let $S_k$ act on the set of edges by reordering, and $S_N$ act on the set of vertices by reordering. Let $\mathrm{sgn}_k$ be the one-dimensional sign representation of $S_k$.
Given some integer $n$ we then define a graded vector space
\[
 \G_n' := 
\begin{cases}
\bigoplus_{N,k}\left( \R \mathsf{dgra}_{N,k}^{conn, \geq 2}\otimes_{S_N\times S_k\ltimes S_2^k} \sgn_k  \right)[(N-1)n-k(n-1)] & \text{for $n$ even} \\
\bigoplus_{N,k}\left( \R \mathsf{dgra}_{N,k}^{conn, \geq 2}\otimes_{S_N\times S_k\ltimes S_2^k} (\sgn_N\otimes \sgn_2^{\otimes k})  \right)[(N-1)n-k(n-1)] & \text{for $n$ odd}
\end{cases}.
\]
In words, we consider vertices as having degree $-n$ and edges having degree $n-1$, and take care of appropriate sign factors under interchange of the position in the ordering of odd objects. 

The graded vector space $\G_n$ is naturally a dg Lie coalgebra with the Lie cobracket given by subgraph contraction,
\[
 \Delta \Gamma = \sum_{\gamma\subset \Gamma} \pm \Gamma/\gamma \wedge \gamma,
\]
where the sum is over connected subgraphs of $\Gamma$ and $\Gamma/\gamma$ is obtained by contracting $\gamma\subset \Gamma$ to a vertex. The signs are induced by the ordering of odd objects imposed.
For example, for $n$ even the sign is the sign of the permutation that moves the edges in $\Gamma/\gamma$ to the left of those in $\gamma$, relative to the original ordering of edges being given as part of the data $\Gamma$.

We also consider the dual space
\[
 \GC_n' = (\G_n')^*,
\]
which is naturally a graded Lie algebra by duality. Now one easily checks that the element
\beq{equ:deltadef}
\delta =\frac{1}{2} \begin{tikzpicture}[scale=1,
vert/.style={draw,outer sep=0,inner sep=0,minimum size=3,shape=circle,fill},
helper/.style={outer sep=0,inner sep=0,minimum size=3,shape=coordinate},
default edge/.style={draw},
 every loop/.style={out=140, in=50, looseness=.8, distance=.8cm }]
\node (v0) at (0,0.72) [vert] {};
\node (v1) at (0.7,0.72) [vert] {};
\draw[default edge] (v0)--(v1);
\end{tikzpicture}\in \GC_n'
\eeq
is a Maurer-Cartan (MC) element, i.e., $[\delta,\delta]=0$. We impose a differential on $\GC_n'$ by twisting with $\delta$, and dually on $\GG_n'$ by cotwisting. We also define the dg Lie subalgebra (resp. quotient)
\begin{align*}
\GC_n &\subset (\GC_n')^\delta & \G_n &\twoheadleftarrow (\G_n')^\delta
\end{align*}
consisting of 1-vertex irreducible graphs with at least bivalent vertices (resp. the quotient modulo graphs with vertices of lower valency or not 1-vertex irreducible).
Combinatorially the differential on $\G_n$ is given by edge contraction
  \begin{align*}
  d \Gamma &= \sum_{e \text{ edge} }  \pm 
 \underbrace{ \Gamma / e}_{\text{contract $e$} }
%  \begin{tikzpicture}[baseline=-.65ex]
%   \node[circle] (g) at (0,0) {$\Gamma$};
%   \draw[dotted] (1,0) circle (.3); 
%   \node[int] (v) at (1,.2) {};
%   \node[int] (w) at (1,-.2) {};
%  % \node at (1.4,.3) {$\scriptstyle v$};
%   \draw (v) edge (w);
%   \draw (g.north east) edge (v)  edge (w) (g.east) edge (v) (g.south east) edge (v) edge (w) ;
%  \end{tikzpicture}
&
\begin{tikzpicture}[baseline=-.65ex]
\node[int] (v) at (0,0) {};
\node[int] (w) at (0.5,0) {};
\draw (v) edge node[above] {$e$} (w) (v) edge +(-.3,-.3)  edge +(-.3,0) edge +(-.3,.3)
 (w) edge +(.3,-.3)  edge +(.3,0) edge +(.3,.3);
\end{tikzpicture}
&\mapsto
\begin{tikzpicture}[baseline=-.65ex]
\node[int] (v) at (0,0) {};
\draw (v) edge +(-.3,-.3)  edge +(-.3,0) edge +(-.3,.3) edge +(.3,-.3)  edge +(.3,0) edge +(.3,.3);
\end{tikzpicture}
  \end{align*}

Next we recall the definition of the Kontsevich Hopf cooperad $\GG_n$.
Let $$\mathsf{dgra}_{N,M,k}^{adm}$$ be the set of directed graphs with $k$ edges and $N+M$ vertices, satisfying the following condition.
We call the first $N$ vertices external, and the last $M$ internal. Then we require that there are no connected components in the graph consisting only of internal vertices.
There is a natural action of the group $S_N\times S_M\times S_k \ltimes S_2^k$ on the set $\mathsf{dgra}_{N,M,k}^{adm}$ by reordering vertices and edges, and flipping edge directions.
Given again an integer $n$ we define
\[
 \GG_n'(r) = 
\begin{cases}
\bigoplus_{M,k} \left( \R\mathsf{dgra}_{N,M,k}^{adm}\otimes_{S_M\times S_k\ltimes S_2} \sgn_k \right)[Mn-k(n-1)]  & \text{for $n$ even} \\
\bigoplus_{M,k} \left( \R\mathsf{dgra}_{N,M,k}^{adm}\otimes_{S_M\times S_k\ltimes S_2} (\sgn_M\otimes \sgn_2^k) \right)[Mn-k(n-1)]  & \text{for $n$ odd}.
\end{cases}
\]
Taking the quotient by the group effectively forgets the direction of edges and the ordering on the edges and internal vertices.
We illustrate an element of $\GG_n(r)$ as a graph with undirected edges and where the internal vertices are unlabelled black dots and the external vertices are numbered circles, as in the following example.

%\iffalse
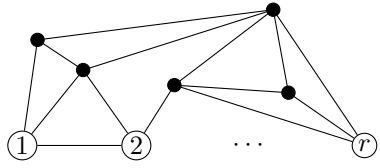
\[\begin{tikzpicture}[scale=1,
vert/.style={draw,outer sep=0,inner sep=0,minimum size=5,shape=circle,fill},
helper/.style={outer sep=0,inner sep=0,minimum size=5,shape=coordinate},
default_edge/.style={draw},
every loop/.style={min distance=10mm,looseness=30}]

\node[ext] (v1) at(0,0){$1$};
%\node (l1) at (0,-0.3){$1$};
\node[ext] (v2) at(1.5,0){$2$};
%\node (l2) at (1.5,-0.3){$2$};
\node (l3) at (3,0){$\cdots$};
\node[ext] (v3) at(4.5,0){$r$};
%\node (l4) at (4.5,-0.3){$r$};

\node (v4) at (0.8,1)[vert]{};
\node (v5) at (3.3,1.8)[vert]{};
\node (v6) at (2,0.8)[vert]{};
\node (v7) at (3.5,0.7)[vert]{};
\node (v8) at (0.2,1.4)[vert]{};

\draw (v1)--(v2);
\draw (v1)--(v4);
\draw (v4)--(v5);
\draw (v5)--(v6);
\draw (v5)--(v7);
\draw (v5)--(v3);
\draw (v6)--(v7);
\draw (v2)--(v4);
\draw (v2)--(v6);
\draw (v3)--(v7);
\draw (v3)--(v6);
\draw (v1)--(v8);
\draw (v4)--(v8);
\draw (v5)--(v8);
\end{tikzpicture}\]

Each space $\GG_n(r)$  is a graded commutative algebra, the product defined by gluing graphs at external vertices.
\beq{equ:graphmult}
      \begin{tikzpicture}[baseline=-.65ex]
  \node[ext] (v1) at (0,0) {1};
  \node[ext] (v2) at (0.5,0) {2};
  \node[ext] (v3) at (1,0) {3};
  \node[ext] (v4) at (1.5,0) {4};
  \node[int] (w1) at (0.5,.7) {};
  \draw (v1) edge (v2) edge (w1) (v2) edge (w1) (v3) edge (w1) ;
 \end{tikzpicture}
 \wedge
     \begin{tikzpicture}[baseline=-.65ex]
  \node[ext] (v1) at (0,0) {1};
  \node[ext] (v2) at (0.5,0) {2};
  \node[ext] (v3) at (1,0) {3};
  \node[ext] (v4) at (1.5,0) {4};
  \node[int] (w2) at (1.0,.7) {};
  \draw   (v2)  edge (w2) (v3)  edge (w2) (v4) edge (w2);
 \end{tikzpicture}
 =
     \begin{tikzpicture}[baseline=-.65ex]
  \node[ext] (v1) at (0,0) {1};
  \node[ext] (v2) at (0.5,0) {2};
  \node[ext] (v3) at (1,0) {3};
  \node[ext] (v4) at (1.5,0) {4};
  \node[int] (w1) at (0.5,.7) {};
  \node[int] (w2) at (1.0,.7) {};
  \draw (v1) edge (v2) edge (w1) (v2) edge (w1) edge (w2) (v3) edge (w1) edge (w2) (v4) edge (w2);
 \end{tikzpicture}
  \eeq
We also consider the dual spaces $\Graphs_n'(r):=\GG_n'(r)^*$.
The spaces $\GG_n'(r)$ furthermore assemble into a Hopf cooperad, with the cocomposition being given by subgraph contraction. For example, the cocomposition corresponding to the tree $T$ as in \eqref{equ:treeex} is
\beq{equ:graphcoop}
 \Delta_T\Gamma = \sum_{\Gamma' \subset \Gamma}\pm (\Gamma/\Gamma')\otimes \Gamma',
\eeq
where the sum is over all subgraphs $\Gamma'\subset \Gamma$ containing the external vertices $i,\dots,j$ and $\Gamma/\Gamma'$ is obtained by contracting $\Gamma'$.
The sign is determined lexicographically by the implicit ordering of odd objects, e.g., the ordering of edges in the case of even $n$.

Next, the spaces $\GG_n'$ carry (co)actions of the graphical Lie (co)algebras $\G_n'$, $\GC_n'$ above.
Concretely, similarly to formula \eqref{equ:graphcoop} we may consider a right $\G_n'$-coaction 
\[
\GG_n' \ni \Gamma \mapsto \sum_{\gamma\subset \Gamma} \pm ( \Gamma/\gamma) \otimes \gamma \in \GG_n' \otimes \G_n',
\]
where now we sum over all subgraphs $\gamma\subset \Gamma$ consisting of only internal vertices.
It is convenient to also consider the dual right action of $\GC_n'$, given by the formula
\beq{equ:bulletdef}
\Gamma \bullet \nu = \sum_{\gamma\subset \Gamma} \pm ( \Gamma/\gamma) \langle \gamma,\nu\rangle,
\eeq
where $\nu\in \GC_n'$ and $\langle-,-\rangle:\G_n'\otimes \GC_n'\to \R$ denotes the duality pairing.
The action $\bullet$ the commutative product and the cooperad structure.

Now, for any coperad $\op C$ the unary cooperations form a coalgebra, hence in particular a Lie coalgebra, and this coalgebra coacts on $\op C$. Dually, $\op C(1)^*$ is a Lie algebra acting on $\op C$.
Abusing notation slightly, we denote the coaction of $\Gamma_1\in \op C(1)^*$ on $\Gamma\in \op C(r)$ by $[\Gamma_1,\Gamma]_\circ\in \op C(r)$. It is given by the formula
\[
[\Gamma_1,\Gamma]_\circ = 
\sum \langle \Gamma_1,\Gamma'\rangle \Gamma''
- \sum \pm \Gamma' \langle \Gamma_1,\Gamma''\rangle
\]
where in each sum we sum over cocompositions into two factors of which one has arity 1 and one arity $r$.

Since the action $\bullet$ respects the cooperad structure and since $[-,-]_\circ$ is defined using only the cooperad structure one can see that both action are compatible in the sense that for $\Gamma_1\in \Graphs_n'(1)$, $\Gamma\in \G_n(r)$ and $\gamma\in \GC_n'$
\beq{equ:action comp0}
[\Gamma_1,\Gamma]_\circ \bullet \gamma = [\Gamma_1,\Gamma \bullet \gamma]_\circ 
+
(-1)^{|\gamma| |\Gamma|}[\Gamma_1\bullet \gamma, \Gamma ]_\circ.
\eeq
Combining both actions $\bullet$ and $[-,-]_\circ$ we hence get an action of the semidirect product $\GC_n'\ltimes \Graphs_n'(1)$ on $\GG_n'$. One can then check that the extension of graded Lie algebras
\[
0 \to \Graphs_n'(1) \to \GC_n'\ltimes \Graphs_n'(1) \to \GC_n' \to 0
\]
splits, and in particular we have a map $\GC_n' \to \GC_n'\ltimes \Graphs_n'(1)$. The piece $\GC_n'\to \Graphs_n'(1)$ sends a graph $\gamma$ to the graph $\gamma_1$ obtained by making one of the vertices external.
We hence arrive at the following action of $\GC_n'$ on $\GG_n'$
\beq{equ:actdef0}
\gamma \cdot \Gamma = [\gamma_1,\Gamma]_\circ + (-1)^{|\gamma||\Gamma|} \Gamma\bullet \gamma.
\eeq
We note that this action respects the cocomposition, but not readily the product. However, the restriction to the 1-vertex irreducible graphs respects the product as well.

We can now use our MC element \eqref{equ:deltadef} to twist the Hopf cooperad $\GG_n'$ to a dg Hopf cooperad $(\GG_n')^\delta$. We finally restrict to the quotient $\GG_n$ obtained by equating to zero graphs with less than bivalent internal vertices
\[
\GG_n \twoheadleftarrow (\GG_n')^\delta.
\]
Finally we arrive at the pair $(\GC_n, \GG_n)$ consisting of a dg Lie algebra acting compatibly on a dg Hopf cooperad.

Let us note that the differential on $\GG_n$ is combinatorially given by edge contraction again
  \begin{align*}
   \delta
   \begin{tikzpicture}[baseline=-.65ex]
    \node[ext] (v) at (0,0) {};
    \node[int](w) at (0,.3) {};
    \draw (v) edge +(-.5,.5) edge +(.5,.5) edge (w) (w) edge +(-.2,.5) edge +(.2,.5);
   \end{tikzpicture}
&=
   \begin{tikzpicture}[baseline=-.65ex]
    \node[ext] (v) {};
    \draw (v) edge +(-.5,.5) edge +(-.2,.5) edge +(.2,.5) edge +(.5,.5);
   \end{tikzpicture}
   &
   \delta
   \begin{tikzpicture}[baseline=-.65ex]
    \node[int] (v) at (0,0) {};
    \node[int](w) at (0,.3) {};
    \draw (v) edge +(-.5,.5) edge +(.5,.5) edge (w) (w) edge +(-.2,.5) edge +(.2,.5);
   \end{tikzpicture}
&=
   \begin{tikzpicture}[baseline=-.65ex]
    \node[int] (v) {};
    \draw (v) edge +(-.5,.5) edge +(-.2,.5) edge +(.2,.5) edge +(.5,.5);
   \end{tikzpicture}\, .
  \end{align*}

\begin{remark}
It has been shown by Kontsevich and Lambrechts-Voli\'c \cite{Kontsevich1999, LambrechtsVolic2014} that $\GG_n$ is a Hopf cooperad model for the little $n$-disks operad for $n\geq 2$, and this model plays a central role in their proof of formality of the little $n$-disks operad.
\end{remark}

\section{Khoroshkin and Willwacher's model for the framed little $n$-disks operad}\label{sec:KWmodel}
The Hopf cooperad model for the framed little disks operad of \cite{KW} is the semidirect product
\[
\GG_n\circ \hat H(G),
\]
for a certain Hopf coaction of the commutative dg Hopf algebra $\hat H(G)$ on $\GG_n$ which we shall now describe.
First, note that a Hopf coaction of $\hat H(G)$ is the same as an $L_\infty$-action of the abelian Lie algebra $\pi^\R(\SO(n))$.
Given the action of $\GC_n$ on $\GG_n$ from the previous section the desired $L_\infty$-action may be encoded by an $L_\infty$-map
\[
\alg g= \pi^\R(\SO(n)) \to \GC_n.
\]
Such a map in turn is equivalent data to a dg Lie algebra map (cf. \eqref{equ:hagdef})
\[
\hag \to \GC_n
\]
and equivalent data to a Maurer-Cartan element 
\[
m\in \bar H(BG)\hat \otimes \GC_n.
\]
The main result of \cite{KW} is the following.
\begin{theorem}[\cite{KW}]\label{thm:KW}
The MC element $m\in H(BG)\hat \otimes \GC_n$ encoding the action of $\SO(n)$ on the little $n$-disks operad (with $n\geq 2$) can be taken to have the form
\beq{equ:mdef}
m=
\begin{cases}
E_n
\begin{tikzpicture}
[scale=1,
vert/.style={draw,outer sep=0,inner sep=0,minimum size=3,shape=circle,fill},baseline=0.3ex]
 \node (1) at (0,0)[vert]{};
 \draw (1) to [out=40,in=0,looseness=1.5] (0,0.4);
 \draw (1) to [out=140,in=180,looseness=1.5] (0,0.4);
\end{tikzpicture} & \text{for $n$ even} \\
   \begin{tikzpicture}[scale=1,
vert/.style={draw,outer sep=0,inner sep=0,minimum size=3,shape=circle,fill},
helper/.style={outer sep=0,inner sep=0,minimum size=3,shape=coordinate},
default edge/.style={draw},
 every loop/.style={out=140, in=50, looseness=.8, distance=.8cm },baseline=-0.5ex]
%\node (l0) at (-3.4,0){$m=$};
\node (l1) at (-1.7,0) {$\underset{k\ge 1}{\sum}\frac{p_{2n-2}^k}{4^k}\frac{1}{2(2k+1)!}$};
\node (l2) at (0,0){$\cdots$};
\node (0) at (0,-0.6) [vert] {};
\node (1) at (0,0.6) [vert] {};
\draw (0) to [out=20,in=-20]  (1);
\draw (0) to [out=40,in=-40]  (1);
\draw (0) to [out=160,in=-160]  (1);
\draw (0) to [out=140,in=-140]  (1);
\node (l3) at (3,0){($2k+1$ edges)};
\end{tikzpicture}
&\text{for $n$ odd}
\end{cases}.
\eeq
Accordingly, a Hopf cooperad model for the framed little disks operad is obtained as 
\[
\GG_n\circ \hat H(G),
\]
with the semidirect product taken with respect to the $\hat H(G)$-action encoded by the above MC element.
\end{theorem}

One shortcoming of this model is that many homotopy biderivations of $\GG_n\circ \hat H(G)$ are not readily realized by actual biderivations.

\section{A new model for the framed little $n$-disks operad}
\subsection{A general construction}\label{sec:gen cons}
Our new Hopf cooperad model for the framed little $n$-disks operad $\LG_n$ is defined similar to $\GG_n$, but with the difference that we put additional decorations on vertices.
We will start with a general construction, using the following data:
\begin{itemize}
\item A commutative Hopf algebra $H$.
\item A commutative Hopf $H$-comodule $K$. (I.e., $K$ has a commutative product and a compatible $H$-comodule structure.)
\end{itemize}

From these data we will construct a Hopf cooperad $\LG_{H,K}'$.
More concretely, we define, as graded vector space
\begin{multline*}
\LG_{n,H,K}'(r) := \\
\begin{cases}
\bigoplus_{M,k} \left( (\R\mathsf{dgra}_{N,M,k} \otimes K^{\otimes M} \otimes H^{\otimes r})\otimes_{S_M\times S_k\ltimes S_2} \sgn_k \right)[Mn-k(n-1)]  & \text{for $n$ even} \\
\bigoplus_{M,k} \left( (\R\mathsf{dgra}_{N,M,k}\otimes K^{\otimes M}\otimes H^{\otimes r})\otimes_{S_M\times S_k\ltimes S_2} (\sgn_M\otimes \sgn_2^k) \right)[Mn-k(n-1)]  & \text{for $n$ odd}.
\end{cases}
\end{multline*}
In words, we consider graphs with two types of vertices as before, but now we additionally decorate every external vertex by an element from $H$ and every internal vertex by an element of $K$.
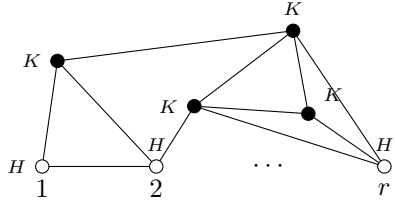
\[\begin{tikzpicture}[scale=1,
vert/.style={draw,outer sep=0,inner sep=0,minimum size=5,shape=circle,fill},
helper/.style={outer sep=0,inner sep=0,minimum size=5,shape=coordinate},
default_edge/.style={draw},
ext/.style={draw,outer sep=0,inner sep=0,minimum size=5,shape=circle},
every loop/.style={min distance=10mm,looseness=30}]

\node[label=180:{$\scriptstyle H$}] (v1) at(0,0)[ext]{};
\node (l1) at (0,-0.3){$1$};
\node[label=90:{$\scriptstyle H$}] (v2) at(1.5,0)[ext]{};
\node (l2) at (1.5,-0.3){$2$};
\node (l3) at (3,0){$\cdots$};
\node[label=90:{$\scriptstyle H$}] (v3) at(4.5,0)[ext]{};
\node (l4) at (4.5,-0.3){$r$};

\node[label=90:{$\scriptstyle K$}] (v5) at (3.3,1.8)[vert]{};
\node[label=180:{$\scriptstyle K$}] (v6) at (2,0.8)[vert]{};
\node[label=30:{$\scriptstyle K$}] (v7) at (3.5,0.7)[vert]{};
\node[label=180:{$\scriptstyle K$}] (v8) at (0.2,1.4)[vert]{};

\draw (v1)--(v2);
\draw (v5)--(v6);
\draw (v5)--(v7);
\draw (v5)--(v3);
\draw (v6)--(v7);
\draw (v2)--(v6);
\draw (v3)--(v7);
\draw (v3)--(v6);
\draw (v1)--(v8);
\draw (v2)--(v8);
\draw (v5)--(v8);
\end{tikzpicture}
\]
Each space $\LG_{n,H,K}'(r)$ is naturally a graded commutative algebra. The product is merely the union of graphs at external vertices as in \eqref{equ:graphmult}, with the additional provision that the decorations at the external vertices $H$ are multiplied using the commutative product on $H$.

Second, the graded vector spaces $\LG_{n,H,K}'(r)$ assemble into a Hopf cooperad $\LG_{n,H,K}'$.
The cocomposition is defined by subgraph contraction as in \eqref{equ:graphcoop}.
\beq{equ:cocomp}
\Delta_T \Gamma = \sum_{\Gamma'\subset \Gamma} \pm (\Gamma/\Gamma') \otimes \Gamma'.
\eeq
However, we need to explain how to handle the additional decorations on the subgraph $\Gamma'$ and the contracted graph $\Gamma/\Gamma'$.
Suppose the internal vertices of $\Gamma'$ are decorated by elements $\kappa_1,\dots ,\kappa_k\in K$, and the external vertices by $h_1,\dots, h_l\in H$. Denote the coactions and coproducts by 
\begin{align*}
\Delta \kappa_j &= \sum k_j'  \otimes \kappa_j'' & \Delta h_j &= \sum h_j' \otimes h_j'',
\end{align*}
with $k_j', h_j', h_j''\in H$ and $\kappa_j''\in K$.
Then the decorations on the vertices of $\Gamma'$ on the right-hand side of \eqref{equ:cocomp} are taken to be $\kappa_1'',\dots \kappa_k'', h_1'',\dots ,h_l''$, and the decoration on the new external vertex of $\Gamma/\Gamma'$ created by contracting $\Gamma'$ is taken to be the product $\pm k_1'\cdots k_k' h_1'\cdots h_l'\in H(G)$.
The sign is chosen lexicographically, relative to the ordering of all objects given.
One can easily check that the above cocomposition morphisms endow the collection $\LG_{n,H,K}'$ with the sructure of a Hopf cooperad. 

Evidently, the construction is functorial in the pair $(H,K)$.
In particular, suppose that we have the following additional datum.
\begin{itemize}
\item A compatible augmentation $K\to \R$.
\end{itemize}
Together with the counit of $H$ this gives us in particular a map of Hopf algebra/comodule pairs $(H,K)\to (\R,\R)$.
Hence, by functoriality, we get a map of Hopf cooperads
\[
\LG_{n,H,K}' \to \LG_{n,\R,\R}' = \GG_n'.
\]
As in section \ref{sec:graphcpx}, for any cooperad $\op C$ the unary operations $\op C(1)$ form a coalgebra which coacts on $\op C$.
Using the above map and duality we hence get an action of the algebra $\GG_n'(1)^* = \Graphs_n(1)'$.
Abusing notation slightly, we denote this action by 
\[
[\Gamma_1 ,\Gamma]_\circ
\]
for $\Gamma_1\in \Graphs_n(1)'$, $\Gamma\in \LG_{n,H,K}'$.

Next suppose that we have in addition the following data:
\begin{itemize}
\item An augmented commutative algebra $A$, with a compatible $K$-action. We consider here $H$ to be equipped with the trivial $A$-action via the augmentation.
\end{itemize}

Then we can formulate a right action on $\LG_{n,H,K}'$ by the Lie algebra $A \hat \otimes \GC_n'$ which generalizes the right action $\bullet$ considered in section \ref{sec:graphcpx}.
For $\Gamma\in \LG_{n,H,K}$, $a\in A$ and $\gamma\in \GC_n'$ the action is given by the formula (cf. \eqref{equ:bulletdef})
\beq{equ:cdotdef}
\Gamma \bullet (a\gamma) = \sum_{\nu\subset \Gamma} \pm (\Gamma/_a \nu) \langle \tilde \nu, \gamma\rangle,
\eeq
where the sum is over subgraphs with only internal vertices, and $\Gamma/_a \nu$ and $\tilde \nu\in \G_n'$ are obtained by the following procedure. Suppose the decorations on vertices of $\nu$ are $\kappa_1,\dots, \kappa_m\in K$. We map all $\kappa_j\in K$ to elements $h_j\in H$ via the composition of coaction and augmentation
\beq{equ:KHmap}
K\to H\otimes K \to H.
\eeq
Then the graph $\Gamma/_a \nu$ is obtained by contracting $\nu$ and decorating the newly formed vertex by $h_1\cdots h_m a$. Furthermore $\tilde \nu$ is the graph obtained from $\nu$ by forgetting decorations.
It is an easy exercise to check that the operations $\Gamma\mapsto \Gamma \bullet (a\gamma)$ indeed form a right action of $A \hat \otimes \GC_n'$ on $\LG_{n,H,K}'$, compatible with the product and cooperadic cocomposition.
Since the action $[-,-]_\circ$ was defined just from the cooperad structure we hence also have the compatibility relation (cf. \eqref{equ:action comp0})
\beq{equ:action comp}
[\Gamma_1,\Gamma]_\circ \cdot (a\gamma) = [\Gamma_1,\Gamma \cdot (a\gamma)]_\circ 
+
(-1)^{\gamma\Gamma}[\Gamma_1\cdot \gamma, \Gamma ]_\circ \epsilon(a),
\eeq
for $\Gamma\in \LG_{n,H,K}'$, $a\gamma\in A\hat \otimes \GC_n'$, $\Gamma_1\in \Graphs_n'$ and $\epsilon:A\to \R$ the augmentation.
This means that the semidirect product $A\hat \otimes \GC_n'\ltimes \Graphs_n'(1)$ acts on $\LG_{n,H,K}'$.
Using the augmentation we can extend the map $\GC_n'\to \GC_n'\ltimes \Graphs_n'(1)$ from section \eqref{sec:graphcpx} to a map
\begin{gather*}
A\hat \otimes \GC_n' \to A\hat \otimes \GC_n'\ltimes \Graphs_n'(1) \\
a\gamma \mapsto (a\gamma, \epsilon(a) \gamma_1).
\end{gather*}
Hence we obtain an action of $A\hat \otimes \GC_n'$ on $\LG_{n,H,K}'$ by the formula (cf. \eqref{equ:actdef0})
\beq{equ:actdef}
(a,\gamma) \cdot \Gamma = 
\epsilon(a) [\gamma_1,\Gamma]_\circ
+(-1)^{|a\gamma| |\Gamma|} \Gamma \cdot (a\gamma).
\eeq

The action is compatible with the cooperadic cocompositions, and compatible with the product if we restrict to 1-vertex irreducible graphs as before.

\subsection{Our model(s)}
From the general construction of the previous subsection we may extract several Hopf cooperads with rich symmetry groups.
We will see below that the Hopf cooperads constructed here are indeed models for the framed little $n$-disks operad.
More concretely, we consider the following two cases
\begin{itemize}
\item Take $H=H(G)$ (with $G=\SO(n)$), take for $K$ the Koszul complex $K=(H(BG)\otimes H(G),d)$ of section \ref{sec:son} and take $A=H(BG)$, with the obvious (Hopf) algebra structures and (co)actions.
This yields a dg Hopf cooperad with a (graded) Lie algebra action
\[
 H(BG)\hat \otimes \GC_n' \actson \LG_{n,H(G),K}' .
\]
We twist by the MC element $\delta+m$ of \eqref{equ:deltadef}, \eqref{equ:mdef} to obtain the dg Lie algebra and Hopf cooperad
\[
 (H(BG)\otimes \GC_n')^{\delta+m} \actson (\LG_{n,H(G),K}')^{\delta+m} .
\]
We then restrict to the subspaces $\GC_n\subset \GC_n'$ of 1-vi graphs all of whose internal vertices are at least bivalent.
Furthermore we take the quotient of $\LG_{n,H(G),K}'$ with respect to graphs which have internal vertices of valence $\leq 1$, where we count nontrivial $K$-decorations as contributing to the valence.
Thus we obtain our first model $\LG_n\twoheadleftarrow (\LG_{n,H(G),K}')^{\delta+m}$ which comes with a dg Lie algebra action
\[
 (H(BG)\hat \otimes \GC_n)^m \actson \LG_{n}.
\]
\item We may repeat the same construction with $H=\hat H(G)$, $K=\hat K$ (see \eqref{equ:hatHdef}, \eqref{equ:hatKdef}) and $A=H(BG)$ and obtain 
a dg Hopf cooperad with a (graded) Lie algebra action
\[
 H(BG)\hat \otimes \GC_n' \actson \LG_{n,H(G),K}'.
\]
Twisting again by $\delta+m$ and restricting/quotienting as before defines our second model $\hLG_n\twoheadleftarrow (\LG_{n,H(G),K}')^{\delta+m}$, which comes with a dg Lie algebra action
\beq{equ:actGChLG}
 (H(BG)\hat \otimes \GC_n)^m \actson \hLG_{n}.
\eeq
Via the quasi-isomorphism $(H(G),K)\to (\hat H(G),\hat K)$ and functoriality we obtain a map of dg Hopf cooperads
\beq{equ:LGhLG}
 \LG_n\to \hLG_n.
\eeq
It is easy to see that this map is a quasi-isomorphism as well.

We finally note that the Lie algebra $\Der(H(BG))$ naturally acts on our input data $(\hat H(G), \hat K)$. Hence, using again functoriality, a larger Lie algebra acts on our model,
\[
 (\Der(H(BG))\ltimes H(BG)\otimes \GC_n)^m \actson \hLG_{n}.
\]
\end{itemize}

The main result of this work is the following.
\begin{theorem}\label{thm:main}
 The dg Hopf cooperad $\LG_{n}$ is quasi-isomorphic to the dg Hopf cooperad model $(\GG_n\circ \hat H(G),d)$ of \cite{KW} for the framed little $n$-disks operad, cf. section \ref{sec:KWmodel}.
\end{theorem}

We shall prove this theorem in section \ref{sec:proof} after some technical preparation.

\subsection{An extension of the construction I}
Before we conduct the proof of Theorem \ref{thm:main} we need the following extension of the construction of the previous two subsections.
First, consider $\LG_{n,H,K}'$ as in section \ref{sec:gen cons}, which comes with an action of $A\hat \otimes \GC_n'$.
Now suppose that $H$ coacts on $\LG_{n,H,K}'$ via operations in $A\hat\otimes \GC_n'$, i.e., there is a map of Hopf algebras
$H^*\to \mU (A\otimes \GC_n')$. Then, similar to \eqref{equ:twistcocomp} we may alter the cooperadic cocomposition in $\LG_{n,H,K}'$ by this coaction.
Concretely, if we denote the original cocomposition \eqref{equ:cocomp} by 
\[
\Gamma\mapsto \sum \Gamma' \otimes \Gamma''
\]
and the coaction $\LG_{n,H,K}' \to H\otimes \LG_{n,H,K}'$ on $\Gamma''$ by
\[
\Gamma''\mapsto \sum h\otimes \Gamma'''
\]
then the twisted cocomposition is defined as 
\[
\Gamma\mapsto \sum (\Gamma'\cup_* h) \otimes \Gamma''',
\]
where the operation $-\cup_* h$ multiplies the decoration of the newly formed vertex in the cocomposition by $h\in H$.
It is easy to check that this prescription indeed defines a cooperadic cocomposition.
Furthermore, if $m\in A\hat\otimes \GC_n'$ is a MC element, then we can construct the twisted cooperad $(\LG_{n,H,K}')^m$, which is acted upon by the twisted dg Lie algebra $\alg k:=(A\hat \otimes \GC_n')^m$
Then, if $f:H^*\to \mU \alg k$ is a Hopf algebra map, then we may apply the same construction to obtain a dg Hopf cooperad
\[
((\LG_{n,H,K}')^m, d+m\cdot, \Delta^f).
\]

We will apply the above construction to the case of the previous subsection, i.e., $\hLG_n$), with $m=\delta$, $H=\hat H(G)$, $K=\hat K$ and $A=H(BG)$ to obtain a Hopf cooperad $(\hLG_n, \delta\cdot , \Delta^\phi)$, where  $\phi:\mU \hag \to \mU \GC_n$ is the map encoded by the MC element of Theorem \ref{thm:KW}.
Note that into this operad we have a direct map to the Khoroshkin-Willwacher model from section \ref{sec:KWmodel}
\beq{equ:map1}
\GG_n \circ \hat H(G) \to (\hLG_n, \delta\cdot , \Delta^\phi),
\eeq
obtained simply by mapping a graph to itself, with all internal vertices decorated by the unit $1\in \hat K$.
This map is easily seen to be a quasi-isomorphism since $H(\hat K)=\R$.

\subsection{An extension of the construction II}
Recall the definition of the action of $A\hat \otimes \GC_n'$ on $\LG_{n,H,K}'$ of formula \eqref{equ:actdef}.
This operation may be extended to an action of $K\hat \otimes \GC_n'$ on the dg Hopf collection, using the same formulas.
More specifically we define a right action of $k\otimes \gamma\in K\otimes \GC_n'$ on $\Gamma\in \LG_{n,H,K}'$ by the formula (cf. \eqref{equ:cdotdef})
\[
\Gamma \bullet (k\gamma) =
\sum_{\nu \subset \Gamma}
\pm (\Gamma/_k \nu) \otimes (\tilde \nu,\gamma),
\]
where on the right we sum over all subgraphs with zero or one external vertices, and multiply $k$ into the decoration of the newly formed vertex. (If the vertex is external, one uses the map $K\to H$ of \eqref{equ:KHmap}.)
The operation $-\bullet (k,\gamma)$ respects the commutative algebra structure. However, it does not respect the cooperadic cocomposition. 

Let us now restrict to the special case of most interest to us, $A=H(BG)$, $H=\hat H(G)$ and $K=\hat K=(H(BG)\otimes \hat H(G), d)$ as before.
Also consider our MC element $m\in H(BG)\hat \otimes \GC_n$. Using the dgca inclusion $H(BG)\subset \hat K$ we can just as well consider $m$ as a Maurer-Cartan element of $\hat K \hat \otimes \GC_n$.
It is easy to see that as such the element $m$ is gauge trivial, simply because $H(\hat K)=\R$.
We shall however construct the explicit gauge transformation trivializing $m$.
To this end, consider the identity element
\[
I\in \Hom(\hat H_\bullet(G),\hat H_\bullet (G))\cong \hat H(G)\hat \otimes \hat H_\bullet(G).
\]
The right-hand side can be considered a Hopf algebra over $\hat H(G)$ and the element $I$ is group-like, reflecting the obvious fact that the identity map respects the coproduct.
We then define the primitive element 
\[
x_u := \log_{\hat H(G)} (I) = -\sum_{j\geq 1} \frac 1 j (1-I)^j \in \hat H(G) \otimes \hat{\alg g}\subset \hat K \otimes \hat{\alg g},
\]
where $\hat{\alg g}$ is the resolution of the real homotopy of $\SO(n)$ as in \eqref{equ:hagdef}.
We also consider the Maurer-Cartan element
\[
m_u = \sum_{w} w^*\otimes w \in H(BG) \otimes \hat{\alg g}\subset \hat K \otimes \hat{\alg g},
\]
where $w$ ranges over a basis of the generators $\bar H^\bullet(BG)$, while $w^*$ stand for the corresponding dual basis elements. In other words, $m_u$ represents the inclusion of the generators into $\hat{\alg g}$.
Consider the dg Lie algebra map $\phi:\hat{\alg g}\to \GC_n$ encoded by $m$. Then clearly
\[
\phi(m_u) = m,
\]
where we (abusively) use the notation $\phi$ also for its $H(BG)$-linear extension.
Hence any gauge transformation trivializing $m_u$ can be mapped to one trivializing $m$ via $\phi$.
\begin{lemma}
The element $x_u\in \hat K \otimes \hat{\alg g}$ gauge trivializes the Maurer-Cartan element $m_u$ in the sense that
\beq{equ:mtriv}
m_u = \frac{e^{-\ad_{x_u}}-1}{-\ad_{x_u}} dx,
\eeq
where $\ad_{x_u}(-)=[x_u,-]$ denotes the adjoint action as usual.
\end{lemma}
\begin{proof}
We first note that \eqref{equ:mtriv} can be conveniently rewritten in the universal enveloping algebra $\hat K\hat \otimes \hat H_\bullet(G)$ in the form
\[
m_u = e^{-x_u} d(e^{x_u}).
\]
Note that the ("identitiy") element $I=e^{x_u}$ has the form 
\[
I=\sum_w w^* \otimes w,
\]
where $w$ ranges over a basis of $\hat H_\bullet(G)$, i.e., words in letters from $\bar H_\bullet(BG)[1]$, and $w^*$ are the dual basis elements, which we similarly identify with words in $\bar H(BG)[1]$.
The differential solely acts on $\hat K$, and if we consider a word $w_1^*\cdots w_k^*\in \hat H(G)\subset \hat K$, then 
\[
d (w_1^*\cdots w_k^*) = w_1^*\cdots w_{k-1}^*\otimes w_k^* \in \hat K=\hat(BG)\hat\otimes H(BG).
\]
(In words, the last factor is moved over to $H(BG)$.)
Comparing with the definition of $m_u$ it is thus clear that 
\[
d e^{x_u} = e^{x_u} m_u,
\]
and we are done.
\end{proof}

Applying the map $\phi:\hat{\alg g}\to \GC_n$ from above we hence obtain an element 
\[
x:=\phi(x_u) 
\]
which gauge-trivializes the MC element $m$ in the dg Lie algebra $\hat K\hat\otimes \GC_n$.
This means in particular, that by the action of the latter Lie algebra introduced in the beginning of this subsection we have an isomorphism of dg Hopf collections
\[
e^{x\cdot} : (\LG_{n}, \delta\cdot) \to (\LG_{n}, \delta\cdot + m\cdot).
\]
We note again that this map is not compatible with the cooperadic cocompositions, in a way we shall study next.
To understand the (non-)compatibility, let us suppose that $k$ is an element of $\hat{\alg g}$ and $\gamma\in \GC_n$.
Then one can check that 
\beq{equ:cococompat}
\Delta (\Gamma\bullet (k\gamma))  - \sum (\pm \Gamma'\bullet (k\gamma)\otimes \Gamma'' + \pm \Gamma'\otimes \Gamma''\bullet (k\gamma))
=
  \sum (\pm \Gamma'\cup_* k \otimes \gamma \cdot \Gamma''.
\eeq
In words, the failure to commute with the cocomposition is just the infinitesimal version of the twist of the cocomposition considered in the previous subsection. Next apply the map $\phi$ from above to our element $I$ to obtain the element 
\[
X:=\phi(I)=e^x \in \hat H(G) \hat \otimes \mU\GC_n \subset \hat K \hat \otimes \mU\GC_n
\]
representing the map $\hat H_\bullet(G)=\mU\hat{\alg g}$ integrating our morphism $\hat{\alg g}\to \GC_n$.
Then the non-infinitesimal statement of formula \eqref{equ:cococompat} is then that the action $X\cdot$ intertwines the twisted and untwisted cocompositions, and hence yields an isomorphism of dg Hopf cooperads
\beq{equ:Xdef}
X\cdot : (\LG_{n}, \delta\cdot, \Delta^\phi) \to (\LG_{n}, \delta\cdot + m\cdot, \Delta).
\eeq

\subsection{Proof of Theorem \ref{thm:main}}\label{sec:proof}
To show the theorem it just remains to concatenate the quasi-isomorphisms defined above into the zigzag of quasi-isomorphisms of dg Hopf cooperads
\beq{equ:zigzag}
(\LG_n, d + \delta\cdot + m\cdot,\Delta) 
\rightarrow (\hLG_n, d + \delta\cdot +m\cdot,\Delta)
\xrightarrow{X\cdot} (\hLG_n, d + \delta\cdot,\Delta^\phi)
\xleftarrow{} \GG_n\circ \hat H(G).
\eeq
Here the left-most map is the quasi-isomorphism \eqref{equ:LGhLG}, the middle map is the isomorphism \eqref{equ:Xdef}, and the right-hand map is the quasi-isomorphism \eqref{equ:map1}.
\hfill\qed

\begin{remark}
We refer to the thesis of the first author \cite{ErikMSc} for more explicit and detailed computations of several steps of our proof. 
\end{remark}

\subsection{A slight extension, and the action of the full homotopy biderivation Lie algebra}
The models above exhibit large dg Lie algebras of biderivations, which however do not yet exhaust all homotopy biderivations.
The difference is that we need to extend $\GC_n$ in the above formulas to the extension $\R L\ltimes \GC_n$, where the additional generator $L$ acts on graphs by the loop order grading, i.e., for $\gamma\in \GC_n$ a graph of look order $k$
\[
[L,\gamma] = k \gamma.
\] 
Note that $\R L\ltimes \GC_n$ also acts on $\GG_n'$, where the additional generator $L$ acts on a graph $\Gamma\in \GG_n'$ via multiplication by a scalar
\beq{equ:Ldef}
L\Gamma = (\#\text{edges}-\# \text{internal vertices}) \Gamma.
\eeq
The same formula leads to a biderivation $L$ on the Hopf cooperads $\LG_{n,H,K}'$ of section \ref{sec:gen cons}.

Next consider a slightly altered graph complex $\LLG_{n,H,K}'$, in which all vertices are additionally decorated by one nonnegative integer, which we call the L-number of that vertex.
We extend the cooperadic composition of \eqref{equ:cocomp} so that
\[
\Delta_T \Gamma = \sum_{\gamma\subset \Gamma} \pm \Gamma/\gamma \otimes \gamma.
\]
where in addition the newly formed vertex of $\Gamma/\gamma$ is decorated by the L-number
\[
 \sum (\text{L-numbers decorating vertices of $\gamma$})
+(\#\text{edges of $\gamma$}) - (\#\text{internal vertices of $\gamma$}).
\]
In words, the L-number remembers the total (internal) loop order of the contracted subgraph at that point.

The commutative algebra structure on $\LLG'_{n,H,K}(r)$ is again given by gluing graphs at external vertices, adding up the L-numbers at external vertices.
There is a forgetful map of Hopf cooperads
\[
 \LLG_{n,H,K}'\to \LG_{n,H,K}'
\]
by forgetting the $L$-numbers of all vertices.
It follows in particular that the action $[,]_\circ$ of the Lie algebra $\Graphs_n'(1)$ (see section \ref{sec:gen cons}) extends to $\LLG_{n,H,K}'$.

We may also extend the right action $\bullet$ of \eqref{equ:cdotdef} to our Hopf cooperad $\LLG_{n,H,K}'$.
To this end we merely declare that in the contracted graph $\Gamma/\nu$ on the right-hand side of \eqref{equ:cdotdef} the newly formed vertex carries L-number
\[
 \sum (\text{L-numbers decorating vertices of $\nu$})
+(\#\text{loops of $\nu$}).
\]
Again the L-number formally remembers the loop order of the contracted piece. Note also that the term $\tilde \nu$ on the right-hand side of \eqref{equ:cdotdef} shall now be obtained by forgetting also all L-numbers on vertices.
One can check that this still defines a right $\GC_n'$-action. Furthermore, the compatibility relation \eqref{equ:action comp} continues to hold.
Hence we may again combine both action $[,]_\circ$ and $\bullet$ to one action of $\GC_n'$ on $\LLG_{n,H,K}'$ using formula \eqref{equ:actdef}.

Furthermore, the biderivation $L$ defined in \eqref{equ:Ldef} continues to exist on $\LLG_{n,H,K}'$, with the same definition. (Note that the operator $L$ ignores the L-numbers.)
For $\gamma\in\GC_n'$ of loop order $l$ and $\Gamma\in \LLG_{n,H,K}'$ we readily verify that
\[
 L(\gamma\cdot \Gamma) - \gamma \cdot (L\Gamma) = - l \gamma \cdot \Gamma,
\]
since the operator $\gamma\cdot$ reduces the first Betti number by $l$.

The main advantage of $\LLG_{n,H,K}'$ over $\LG_{n,H,K}'$ is now that for any $a\in \bar A$ we may define an additional biderivation which we shall denote (abusively) $aL$. It is defined on a graph $\Gamma\in \LLG_{n,H,K}'$ defined such that
\[
 (aL) \Gamma = \sum_{v\text{ int. vertex}}L(v) (a\cup_v \Gamma),
\]
where the sum runs over internal vertices $v$, $L(v)\in \mathbb{N}_0$ is the L-number of $v$ and $(a\cup_v \Gamma)$ is the graph $\Gamma$, unaltered except that the decoration (in $K$) at $v$ is decorated by $a$.
A quick computation shows that for $\gamma\in \GC_n'$ of loop order $l$ one has the commutator relation
\[
 (aL)(\gamma\cdot \Gamma) - \gamma \cdot ((aL)\Gamma) = - l (a\gamma) \cdot \Gamma,
\]

Hence the action of $A\hat\otimes\GC_n'$ extends to an action
\[
  A\hat\otimes \left(\R L\ltimes \GC_n' \right) \actson \LLG_{n,H,K}'.
\]
Now we restrict to $H=\hat H(BG)$, $K=\hat K$ and $A=H(BG)$. 
We additionally twist by the Maurer-Cartan element $\delta+m$ as before and to obtain an action 
\[
 (\Der(H(BG))\ltimes H(BG)\otimes \left(\R L\ltimes \GC_n' \right))^{\delta+m} \actson (\LLG_{n,H,K}')^{\delta+m}.
\]
Finally, we restrict to the pieces where all graphs have all internal vertices of valence $\geq 2$, counting a decoration by L-number $>0$ as a valence.
This yields a dg Hopf cooperad acted upon by a dg Lie algebra
\beq{equ:mainaction}
  (\Der(H(BG))\ltimes H(BG)\otimes \left(\R L\ltimes \GC_n \right))^m \actson \LLG_{n}.
\eeq
We have the following result.
\begin{proposition}
 The map $\LLG_{n}\to \LG_n$ by forgetting the L-numbers is a quasi-isomorphism.
\end{proposition}
\begin{proof}
% We call an L-antenna an internal vertex with one incident edge, no $K$-decoration, and (necessarily) positive L-number.
%We equip $\LLG_{n}$ with the filtration by non-L-antenna vertices and consider the corresponding spectral sequence. The first differential contracts L-antennas
%\[
%... pic .
%\]
The proof is similar to \cite[Proposition 3.4]{Willwacher2014} and is left to the reader.
%The essential piece of the differential which allows to remove L-numbers is this 
%\[
 %... pic
%\]
\end{proof}

Together with \eqref{equ:mainaction} and our main Theorem we hence have found a Hopf cooperad model $\LLG_n$ for the framed little $n$-disks operad that is acted upon by a large dg Lie algebra. In fact, it is shown in \cite{BW} that the dg Lie algebra on the left-hand side of \eqref{equ:mainaction} is quasi-isomorphic to the homotopy (bi)derivations of (the cochain model of) that operad.

% cf. documentation
\makeatletter
\providecommand\@dotsep{5}
\makeatother

\printbibliography

\end{document}